\documentclass{ifacconf}

\usepackage[dvipsnames]{xcolor}
\definecolor{my_gray}{RGB}{128,128,128}
\definecolor{my_red}{RGB}{255,44,0}
\definecolor{my_cyan}{RGB}{135, 204, 238}
\definecolor{my_indigo}{RGB}{43, 33, 118}
\definecolor{my_rose}{RGB}{204, 102, 118}
\definecolor{my_rose_2}{RGB}{247,231,232}

\usepackage{arydshln}
\usepackage{fancyhdr}

\usepackage{amsmath,bm}
\usepackage{amsfonts}
\usepackage{mathtools}
\usepackage{tikz}
\usepackage{graphicx}
\usepackage{adjustbox}

\usepackage{pgfplots}
\pgfplotsset{compat=1.18}
\usepackage{float}
\usepgfplotslibrary{groupplots}
\newlength\fwidth
\usepackage{xcolor,dashrule}
\newcommand{\solidbox}[2][1.5em]{\textcolor{#2}{\rule[0.55ex]{#1}{1.6pt}}}
\newcommand{\dashboxx}[2][1.5em]{\textcolor{#2}{\hdashrule[0.7ex]{#1}{1.6pt}{1mm 0.5mm}}}

\makeatletter
\newcommand\bigDiamond{\mathop{\mathpalette\bigDi@mond\relax}}
\newcommand\bigDi@mond[2]{%
  \vcenter{\hbox{\m@th
      \scalebox{\ifx#1\displaystyle 2\else1.2\fi}{$#1\Diamond$}%
    }}%
}
\newcommand\bigLozenge{\mathop{\mathpalette\bigL@zenge\relax}}
\newcommand\bigL@zenge[2]{%
  \vcenter{\hbox{\m@th
      \scalebox{\ifx#1\displaystyle 2\else1.2\fi}{$#1\blacklozenge$}%
    }}%
}

\usetikzlibrary{positioning, arrows.meta, quotes, calc,fit, shadows.blur}
\tikzstyle{box} = [rounded corners = 1mm, fill = gray!2, blur shadow={shadow blur steps=100}]

\usepackage{accents}

\let\epsilon\varepsilon

\newcommand{\bbC}{{\mathbb C}}

\newcommand{\bbR}{{\mathbb R}}

\newcommand{\cH}{\mathcal{H}}

\newcommand{\cP}{\mathcal{P}}

\makeatletter
\newcommand*\RedeclareTheoremEnv[1]{%
  \expandafter\let\csname #1\endcsname\relax
  \expandafter\let\csname end#1\endcsname\relax
}

\usepackage{algorithm,algpseudocode}

\makeatletter
\@ifundefined{alg}{}{%
   
  \theoremstyle{plain}\newtheorem{alg}{Algorithm}[section]}
\makeatother

\definecolor{redBright}{RGB}{255,44,0}
\definecolor{cyanLight}{RGB}{135,204,238}
\definecolor{indigoDeep}{RGB}{43,33,118}
\definecolor{roseMuted}{RGB}{204,102,118}

\definecolor{gray050}{RGB}{249,249,249}
\definecolor{gray200}{RGB}{236,236,236}
\definecolor{gray400}{RGB}{176,176,176}
\definecolor{gray600}{RGB}{122,122,122}

\definecolor{redFirebrick}{RGB}{232,103,37}

\definecolor{colorGA}{RGB}{0,0,0}
\definecolor{colorGB}{RGB}{123,0,0}
\definecolor{colorGC}{RGB}{198,26,26}
\definecolor{colorGD}{RGB}{232,115,90}
\definecolor{colorGE}{RGB}{68,119,170}
\definecolor{colorGray}{RGB}{122,122,122}
\definecolor{colorInf}{RGB}{236,236,236}
\usepackage{natbib}
\usepackage{glossaries}

\begin{document}
\newacronym{les}{LES}{Local Exponential Stability}
\newacronym{pbh}{PBH}{Popov-Belevitch-Hautus}
\newacronym{roa}{ROA}{region of attraction}
\newacronym{cps}{CPS}{Cyber-Physical System}

\begin{frontmatter}
    \vspace*{-2.0em}
    \begin{center}
        \parbox{0.95\textwidth}{\centering \itshape \small
            \textcopyright 2026 Zeyad M. Manaa, Nathan van de Wouw, and Michelle S. Chong. This work has been accepted to IFAC World Congress 2026 for publication under a Creative Commons Licence CC-BY-NC-ND.
        }
    \end{center}

    \title{Confidentiality of Linear Control Systems with Quadratic Output Under Sensor Attacks [Extended Version]}

    \author[TUE]{Zeyad M. Manaa}
    \author[TUE]{Nathan van de Wouw}
    \author[TUE]{Michelle S. Chong}

    \address[TUE]{Department of Mechanical Engineering,
        Eindhoven University of Technology, The Netherlands\\
        e-mail: \{z.manaa; n.v.d.wouw;  m.s.t.chong\}@tue.nl.}

    \begin{abstract}
        We study the state estimation problem for linear control systems with quadratic outputs which are locally unobservable at the equilibrium. We show that, despite this inherent lack of observability, an adversary with sensor read and write capability can induce observability by injecting an appropriate signal into the measurement channel. Taking the role of an adversary, we characterize when an injected signal can or cannot induce observability and, in the successful case, construct an observer that achieves local exponential convergence of state estimates to the true states of the system. A simulation study demonstrates our results.
    \end{abstract}

\end{frontmatter}

\section{Introduction}
Modern \glspl{cps} rely on networked sensors for monitoring and regulating the physical systems. This cyber connectivity improves functionality but also exposes new risks at the measurement layer, which can be exploited maliciously. See \citep{chong2019tutorial, dibaji2019systems} for comprehensive surveys. An adversary capable of reading and writing sensor data might be able to infer the internal state of the system. With such information, they can launch a more dangerous attack, which disrupts the normal performance of the system, while staying undetected \citep{dibaji2019systems}. This paper shows how an adversary can overcome a system's inherent lack of observability to reveal its internal state through a stealthy sensor attack.

Among other attack strategies, sensor attacks in \gls{cps} have gained substantial attention. These attacks aim to manipulate measurement data while remaining undetected. In this setting, studies by \cite{cardenas2011attacks,murguia2019model, guo2018worst} explored undetectable attack strategies. Underlying the attack strategies mentioned above is the assumption that the attacker knows the controller state. We study the validity of this assumption by examining both plant and controller state availability for linear systems with quadratic outputs. This class of systems arises in various applications. The control of quadratic outputs for such systems is studied in \citep{montenbruck2017linear}, where the authors were motivated by mechanical and Hamiltonian systems in which specific energy functions are to be regulated; see also the references therein. Related output models also arise in range-based robot localization, and terrain-aided navigation, which are discussed in \citep[and references therein]{Berkane2023}. Such measurements can be  communicated over networked sensing channels vulnerable to an adversary with sensor read and write access, motivating the confidentiality analysis for this class of systems.

In this paper, we adopt the role of the adversary by designing sensor attacks to infer the unmeasured states of the control system. We assume that the adversary can read from and write to the sensor, but has neither read nor write access to the control input. {This setting has been addressed} in the literature (see for example the work of \cite{umsonst2021confidentiality} and \cite{chong2022controller}). This problem setting is particularly interesting since the control input channel is typically assumed to be protected and inaccessible to the attacker (since the manipulation of control signals can cause immediate physical effects), while the sensing channel is often left more vulnerable. We study a linear plant with a quadratic sensor output map which is stabilized by an output-feedback-based dynamic controller. We assume that the closed-loop system exhibits inherent stability characteristics. However, the attack-free control system is unobservable at the equilibrium due to the quadratic nature of the output map. Specifically, we pose the question whether the adversary can nevertheless estimate the states of the plant and the controller. The goal of the adversary then is to induce observability by manipulating the sensor output channel while simultaneously guaranteeing the preservation of the closed-loop system's stability to remain undetected. Figure \ref{fig:high_level_problem_setup} illustrates an overview of the problem setup.

\begin{figure}
    \centering
    \resizebox{\linewidth}{!}{%
        \begin{tikzpicture}[scale=1, every node/.style={transform shape=false}]
  \pgfdeclarelayer{background}
  \pgfsetlayers{background,main}

  \def\vGap{0.8}
  \def\hStep{2}
  \def\advPos{4.3}
  \def\labelShift{0.4}
  \def\estDrop{1.2}

  \def\boxTop{1.0}
  \def\boxBottom{-2.3}
  \def\boxSep{-1.8}

  \tikzset{
  signal/.style={
  line width=0.8pt,
  -{Latex[round, length=3mm]}
  },
  block/.style={
      draw,
      rounded corners=0.8pt,
      minimum width=10mm,
      minimum height=8mm,
      line width=0.8pt,
      align=center,
      fill=white,
      font=\normalsize
    },
  container/.style={
      box,
      thick,
      rounded corners=0.8pt,
      draw=gray!80
    }
  }

  \tikzset{
    control container/.style={
        box,
        thick,
        rounded corners=0.8pt,
        draw=gray!80,
        fill=gray200!40
      },
    adversary container/.style={
        box,
        thick,
        rounded corners=0.8pt,
        draw=gray!80,
        fill=my_rose_2
      }
  }

  \node[block] (cont) at (0,0) {$\Sigma_c$};
  \node[block] (plant) at (\hStep, 0) {$\Sigma_p$};
  \node[block] (estimator) at (\advPos, 0) {$\Sigma_{\textrm{obs.}}$};

  \node[
    draw,
    circle,
    fill=white,
    minimum size=0.1,
    inner sep=1pt,
    line width=1pt
  ] (sum) at ($(plant.south)+(0,-\vGap)$) {$+$};

  \draw[signal]
  ($(cont.north)$) --
  ++(0,\labelShift) --
  ($(plant.north) + (0,\labelShift)$) --
  (plant.north);

  \draw[signal]
  (plant.south) --
  (sum.north);

  \draw[signal]
  (sum.west) --
  ($(cont.south)+(0,-\vGap)$) --
  (cont.south);

  \coordinate (a_top)  at (3.4, 0);
  \coordinate (a_drop) at (a_top |- sum);
  \draw[signal]
  (a_top) node[above, font=\normalsize]{$a$} --
  (a_drop) --
  (sum.east);

  \coordinate (y_tilde_branch) at ($(cont.south)+(0.6,-\vGap)$);
  \fill (y_tilde_branch) circle (1.5pt);

  \draw[signal]
  (y_tilde_branch) |-
  ($(estimator.south)+(0,-\estDrop)$) --
  (estimator.south);

  \draw[signal] (estimator.east) -- ++(1.2, 0)
  node[right, align=left, font=\normalsize] {Estimated\\states of both\\$\Sigma_p$ \& $\Sigma_c$};

  \begin{pgfonlayer}{background}

    \coordinate (CS_BL)  at ($(cont.west)+(-0.30,\boxBottom)$);
    \coordinate (CS_TR)  at ($(plant.east)+(0.30,\boxTop)$);
    \coordinate (CS_Sep) at ($(CS_BL |- 0,\boxSep)$);

    \fill[gray200!40]
    (CS_BL) rectangle (CS_Sep -| CS_TR);

    \fill[gray200!80]
    (CS_Sep -| CS_BL) rectangle (CS_TR);

    \draw[control container]
    (CS_BL) rectangle (CS_TR);

    \draw[thick, draw=gray!80]
    (CS_Sep -| CS_BL) -- (CS_Sep -| CS_TR);

    \node[font=\normalsize]
    at ($(CS_BL)!0.5!(CS_Sep -| CS_TR)$)
    {Control System};

    \coordinate (Adv_BL)  at ($(estimator.west)+(-0.60,\boxBottom)$);
    \coordinate (Adv_TR)  at ($(estimator.east)+(0.70,\boxTop)$);
    \coordinate (Adv_Sep) at ($(Adv_BL |- 0,\boxSep)$);

    \fill[my_rose_2!50, fill opacity=0.8]
    (Adv_BL) rectangle (Adv_Sep -| Adv_TR);

    \fill[my_rose_2!80, fill opacity=0.2]
    (Adv_Sep -| Adv_BL) rectangle (Adv_TR);

    \draw[adversary container]
    (Adv_BL) rectangle (Adv_TR);

    \draw[thick, draw=gray!80]
    (Adv_Sep -| Adv_BL) -- (Adv_Sep -| Adv_TR);

    \node[font=\normalsize]
    at ($(Adv_BL)!0.5!(Adv_Sep -| Adv_TR)$)
    {Adversary};

  \end{pgfonlayer}
\end{tikzpicture}
    }
    \caption{Block diagram of the closed-loop system with controller \(\Sigma_c\) and plant \(\Sigma_p\), subject to an adversarial sensor attack \(a\). The adversary uses the probed measurements to construct an estimate of both \(\Sigma_p\) and \(\Sigma_c\) states.}
    \label{fig:high_level_problem_setup}
\end{figure}
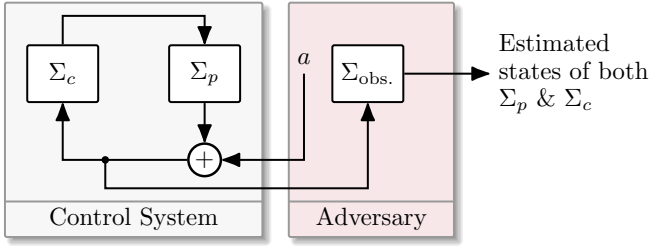

Existing studies address related but distinct problems. For example, \cite{chong2022controller} studies general nonlinear systems that are originally semi-globally asymptotically stable. They show that if the system is detectable, an adversary can passively estimate both the plant and the controller state. On the other hand, if the plant is not detectable, the core approach in that work employs a time-shared probing strategy. Herein, stealthiness is maintained by periodically injecting an observability-inducing attack signal through the sensor channel for a sufficiently long duration, to reconstruct the controller’s state, and then turning it off so that the closed-loop system remains semi-globally practically stable thereby allowing the adversary to avoid detection. While this framework demonstrates the feasibility of breaching controller confidentiality under such attacks, it does not provide a constructive procedure for designing the attack signal.

As described earlier, the attack design problem can be viewed as the problem of input design to induce observability. This general problem has been studied in~\cite{Berkane2023}, who develop a Kalman-type observer for linear systems with multiple quadratic outputs. They derive conditions on the input signal to induce uniform observability, which, in turn, allows a Kalman-like observer with global exponential convergence of the estimation error. Although not related to attack design, this approach represents a possible strategy that could be leveraged to address the state estimation problem discussed earlier. However, the results of~\cite{Berkane2023} impose excitation conditions only to guarantee uniform observability. They do not constrain the input signal so that the resulting closed-loop dynamics remain stable. Similarly, the work by \cite{nesic2000output} on Wiener systems is also of important relevance. Although it is framed as a stabilization problem (with a three-mode controller), its first two modes require an implicit state reconstruction step before the final stabilization mode takes place. Also, as in the previous case, stability is not a required property for the state reconstruction step.

Our work is distinct from these approaches as we present a \textit{systematic method to design an attack signal}, for a specific class of nonlinear systems, which \textit{acts only on the sensor output channel} and that \textit{preserves the closed-loop stability}. The attacker has no access to the control input and aims to estimate the unmeasured states of the plant and controller. To this end, the attacker designs an observer with local exponential convergence properties while also ensuring that the designed attack preserves closed-loop stability to avoid detection. This systematic attack design procedure therefore elucidates the inherent trade-off between estimation performance and stealthiness.

In this paper, we provide rigorous analysis to support the design of an additive attack signal, using the measured output and a computable gradient of the quadratic sensor map. We identify the conditions under which such an injected signal recovers local observability, i.e., the linearized system at the equilibrium is observable. This attack-induced observability property enables the design of a Luenberger-type observer that is capable of estimating both the plant and controller states from the compromised measurements. Since the designed attack signal is added to the closed-loop system, it can reveal the attacker’s presence by compromising stability or triggering anomaly detectors. To remain stealthy, the proposed design ensures that the induced attack preserves the inherent stability of the closed-loop system.

The main contributions of this paper are as follows:
\begin{itemize}
    \item We show that a closed-loop system with a quadratic sensor output map, though locally unobservable at the equilibrium, can be rendered observable through a properly designed, additive attack signal while preserving the closed-loop stability of the system at hand.
    \item We characterize the conditions under which an injected signal can or cannot induce observability.
    \item We design a Luenberger-type observer that allows the adversary to estimate both plant and controller states with guaranteed local exponential convergence of the state estimate and an explicit estimate of the \gls{roa} for both the state and the estimation error.
\end{itemize}

The remainder of the paper is organized as follows. Section \ref{sec:problem_formulation} formulates the problem and defines the adversary’s objectives. Next, Section \ref{sec:inducing_observability} presents the conditions and procedure for inducing local observability. Section \ref{sec:observer} introduces the observer design, analyzes the preservation of stability, and derives the \gls{roa} for both the state and the estimation error. Section \ref{sec:example} demonstrates the results with an illustrative example. Finally, Section \ref{sec:conc} concludes the paper. Before starting Section \ref{sec:problem_formulation}, we first introduce the notation used in this paper.

We denote the set of real (complex) numbers by $\bbR$ ($\bbC$). The set of real (complex) vectors of dimension $n$ is denoted as $\bbR^n$ ($\bbC^n$). By ${(\cdot)}^\top$, we denote the transpose of a vector or a matrix. In addition, we use $I_n$ to denote the identity matrix of dimension $n$. The spectrum of a matrix \(A\) is represented by \(\sigma(A)\). If \(A\) is a symmetric matrix, the notation \(A \succ 0\) (resp., \(A \succeq 0\)) and \(A \prec 0\) (resp., \(A \preceq 0\)) means that \(A\) is positive definite (resp., positive semidefinite) and negative definite (resp., negative semidefinite), respectively. {A system $\dot{x}=f(x)$ is {locally exponentially stable} (LES) if $\exists \Delta, c, k > 0$ such that $ \|x(t)\| \leq k e^{-ct}\|x(0)\|, \; \forall t \geq 0, \ \forall \|x(0)\| \leq \Delta.$}

\section{Problem Formulation} \label{sec:problem_formulation}

\subsection{Problem setup and challenges}
\label{sec:problem_formulation_challenge}
We consider a linear plant $\Sigma_p$ with state \(x_p \in \bbR^{n_p}\),
scalar output $y$, and a quadratic sensor output map with \(Q_p=Q_p^\top\):
\begin{align}\label{eq:plant}
  \Sigma_p: & \quad
  \begin{cases}
    \dot{x}_p = A_p x_p + B_p u, \\
    y = x_p^\top Q_p x_p
  \end{cases}
\end{align}
with a controller $\Sigma_c$, defined by
\begin{align} \label{eq:controller}
  \Sigma_c: & \quad
  \begin{cases}
    \dot{x}_c = A_c x_c + B_c y, \\
    u = C_c x_c + D_c y,
  \end{cases}
\end{align}
where \(x_c \in \bbR^{n_c}\) is the controller state.
We write the closed-loop system consisting of $\Sigma_p$ and $\Sigma_c$ in feedback compactly by introducing $z:= [{x}_p^\top, {x}_c^\top]^\top \in \bbR^{n}$, where $n = n_p + n_c$, which has the following dynamics:
\begin{equation} \label{eq:sys_dynamics_clean}
  \dot{z} = Az + Bh(z) =: f(z), \qquad h(z) = z^\top Q z
\end{equation}
where \(A:=\begin{bsmallmatrix} A_p & B_pC_c \\ {0} & A_c \end{bsmallmatrix}\), \({B}:=\begin{bsmallmatrix} B_p D_c \\ B_c \end{bsmallmatrix} \) and \(Q := \begin{bsmallmatrix}
  Q_p & {0}\\ {0} & {0}
\end{bsmallmatrix}.\)

In this paper, we adopt the following assumption.

\begin{assumption}[Stability] \label{as:system_stability}
  The  equilibrium ${z}={0}$ of the closed-loop system \eqref{eq:sys_dynamics_clean} is LES.
\end{assumption}

The challenge of estimating the state $z$ of system \eqref{eq:sys_dynamics_clean} lies in the quadratic nature of the output map $h(z)$, which renders the system \eqref{eq:sys_dynamics_clean} locally unobservable at the origin. In particular, the linearized system at the origin defined by the pair $(F,H)$ below,
\begingroup
\begin{equation} \label{eq:lin_dynamics}
  \hspace{-1em}
  F := \frac{\partial f(z)}{\partial z}\bigg|_{z=0} = A, \quad
  H := \frac{\partial h}{\partial z}\bigg|_{z=0} = (2z^\top Q )|_{z=0} = 0 \hspace{-0.5em}
\end{equation}
\endgroup
is not observable. This renders observers which rely on local observability such as the Kazantzis-Kravaris-Luenberger (KKL) observers \citep{Kazantzis1998,luenberger1966observers} ineffective.

\begin{remark} \label{rem:system_stability}
  Equation \eqref{eq:lin_dynamics} shows that Assumption~\ref{as:system_stability} holds if and only if $A$ is Hurwitz.
\end{remark}

\subsection{The proposed approach}
Given the fact that the origin of the closed-loop system \eqref{eq:sys_dynamics_clean} is LES by design as stated in Assumption \ref{as:system_stability}, one may consider estimating the states $z$ of the closed-loop system \eqref{eq:sys_dynamics_clean} with an open-loop observer. The primary drawback of this approach is that the convergence rate cannot be tuned with an open-loop observer (i.e., a plant-based predictor). However, the adversary is not a passive intruder, but can also manipulate the sensor measurements via the injected attack signal $a$. Thus, using model knowledge, the adversary can potentially induce local observability of the closed-loop system \eqref{eq:sys_dynamics_clean}. This, in turn, allows for more flexible observer design in terms of the tunable convergence rate of the estimation error. In this case, the feedback received by the controller is no longer $y$ but rather given by
\begin{align}\label{eqn:corrupted_signal}
  \tilde{y} = y + a.
\end{align}
The controller $\Sigma_c$ therefore becomes
\begin{align*}
  \dot{{x}}_c = A_c {x}_c + B_c \tilde{y}, \quad u = C_c {x}_c + D_c \tilde{y}.
\end{align*}
Substituting this compromised control law into the plant dynamics yields the following closed-loop system:
\begin{align}
  \dot{{z}} = A {z} + {B} (h({z}) + a)\label{eq:system_dyn} \;=:\tilde{f}(z,a), \quad h({z}) = {z}^\top Q {z}.
\end{align}
The challenge of designing an attack signal $a$ such that \eqref{eq:system_dyn} is rendered locally observable is addressed in this paper. In the following subsection, we discuss, in more detail, the adversary's model.

\subsection{Adversary capabilities and objectives} \label{sec:adversary_model}

We assume that the adversary knows the plant and controller models but does not know the initial conditions of either the plant state, \(x_p(0)\), or the controller state, \(x_c(0)\). The adversary can {eavesdrop on the} sensor measurements and manipulate them; however, the adversary does not have access to the control input \(u\), nor can he/she alter it. Finally, since the adversary knows the output map of the system and the system parameters, the output gradient can be computed using the known model and an estimate of the system's state. In summary, the adversary operates under the following conditions.

\begin{assumption}[Adversary's Capabilities]
  The adversary
  \begin{enumerate}\label{as:adversary_model}
    \item knows the system model $(A, {B}, Q)$,
    \item has access to the true system output $h(z) = {z}^\top Q {z}$ and can manipulate it using the attack signal $a$, leading to the perturbed output $\tilde{y}$ in \eqref{eqn:corrupted_signal},
    \item estimates the output map's gradient using $\nabla h(\hat z)=2Q \hat z$, where \(Q\) is known and \(\hat z\) is an estimate of the state of system \eqref{eq:system_dyn},
    \item does not know the control input $u$, nor the plant and controller initial conditions $x_p(0),\; x_c(0)$, respectively.
  \end{enumerate}
\end{assumption}

Under Assumptions \ref{as:system_stability} and \ref{as:adversary_model}, we take the role of the adversary aiming to infer the states of the closed-loop control system \eqref{eq:system_dyn} consisting of the plant \(\Sigma_p\) and controller \(\Sigma_c\).  This problem setup is illustrated in detail in Figure~\ref{fig:problem_setup}. The adversary’s goals are then formalized as follows

\begin{objective}[Inducing Observability]
  \label{obj:1}
  Design the injected signal \(a\) using the measured output \(h(z)\) and the computable gradient \(\nabla h(\hat z)=2Q\hat z\), such that it renders the system \eqref{eq:system_dyn} locally observable at $z=0$, i.e., {the linearization of the closed-loop system subject to the injected signal $a$ is observable.}
\end{objective}

\begin{objective}[State Observer Design]
  \label{obj:2}
  Design a state observer, whose parameters can be tuned, to obtain an estimate $\hat{z}$ of the state $z$, so that the origin is a locally exponentially stable equilibrium point of the estimation error ($z(t) - \hat{z}(t)$) dynamics,
  with rate $\alpha > 0$; i.e.,
  $
    \|z(t) - \hat{z}(t)\| \le \kappa e^{-\alpha t}\|z(0) - \hat{z}(0)\|
  $
  {for some constant $\kappa > 0$, for all trajectories \(z\) and \(\hat{z}\)} {starting from initial conditions \(z(0)\) and \(\hat{z}(0)\) in a neighborhood of the origin.}
\end{objective}

\begin{objective}[Closed-loop stability]
  \label{obj:3}
  Design the attack signal \(a\) such that the origin remains a locally exponentially stable equilibrium of the closed-loop system \eqref{eq:system_dyn}.
\end{objective}

Objective \ref{obj:3} ensures that the adversary remains hidden, as causing the system to become unstable would trigger anomaly detectors. This objective is related to the stealth property of the adversary.

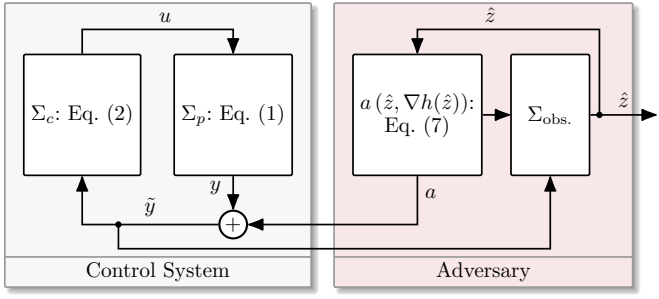
\begin{figure}
  \centering
  \resizebox{\linewidth}{!}{%
    \begin{tikzpicture}[scale=1, every node/.style={transform shape=false}]
  \pgfdeclarelayer{background}
  \pgfsetlayers{background,main}

  \def\vGap{0.8}
  \def\hStep{2.5}
  \def\attackPos{5.55}
  \def\advPos{7.75}
  \def\labelShift{0.4}
  \def\estDrop{1.2}
  \def\zHeight{0.4}

  \def\boxTop{1.85}
  \def\boxBottom{-2.85}
  \def\boxSep{-2.35}

  \tikzset{
  signal/.style={
  line width=0.8pt,
  -{Latex[round, length=3mm]}
  },
  block/.style={
      draw,
      rounded corners=0.8pt,
      minimum width=13mm,
      minimum height=20mm,
      line width=0.8pt,
      align=center,
      fill=white,
      font=\normalsize
    },
  container/.style={
      box,
      thick,
      rounded corners=0.8pt,
      draw=gray!80
    }
  }

  \tikzset{
    control container/.style={
        box,
        thick,
        rounded corners=0.8pt,
        draw=gray!80,
        fill=gray200!40
      },
    adversary container/.style={
        box,
        thick,
        rounded corners=0.8pt,
        draw=gray!80,
        fill=my_rose_2
      }
  }

  \node[block] (cont) at (0,0)
  {$\Sigma_c$: Eq. \eqref{eq:controller}};

  \node[block] (plant) at (\hStep,0)
  {$\Sigma_p$: Eq. \eqref{eq:plant}};

  \node[block] (attack) at (\attackPos,0)
  {$a\left(\hat{z}, \nabla h(\hat z) \right)$: \\
    Eq. \eqref{eq:attack_signal}};

  \node[block] (estimator) at (\advPos,0)
  {$\Sigma_{\textrm{obs.}}$};

  \node[
    draw,
    circle,
    fill=white,
    minimum size=0.1,
    inner sep=1pt,
    line width=1pt
  ] (sum) at ($(plant.south)+(0,-\vGap)$) {$+$};

  \draw[signal]
  (cont.north) --
  ++(0,\labelShift) --
  node[pos=0.55, above] {$u$}
  ($(plant.north)+(0,\labelShift)$) --
  (plant.north);

  \draw[signal]
  (plant.south) --
  node[pos=0.35, left, xshift=-2pt] {$y$}
  (sum.north);

  \draw[signal]
  (sum.west) --
  node[pos=0.5, above] {$\tilde{y}$}
  ($(cont.south)+(0,-\vGap)$) --
  (cont.south);

  \coordinate (aDrop) at (attack.south |- sum.east);

  \draw[signal]
  (attack.south) --
  node[pos=0.35, right] {$a$}
  (aDrop) --
  (sum.east);

  \draw[signal]
  (attack.east) --
  (estimator.west);

  \coordinate (y_tilde_branch) at ($(cont.south)+(0.6,-\vGap)$);
  \coordinate (fb) at ($(estimator.south)+(0,-\estDrop)$);

  \fill (y_tilde_branch) circle (1.5pt);

  \draw[signal]
  (y_tilde_branch) |-
  (fb) --
  (estimator.south);

  \coordinate (zOut) at ($(estimator.east)+(1.1,0)$);
  \coordinate (tap)  at ($(estimator.east)+(0.15,0)$);

  \draw[signal]
  (estimator.east) --
  node[above] {$\hat z$}
  (zOut);

  \draw[signal]
  (tap) |-
  node[pos=0.8, above] {$\hat z$}
  ($(attack.north)+(0,\zHeight)$) -|
  (attack.north);

  \fill (tap) circle (1.5pt);

  \begin{pgfonlayer}{background}

    \coordinate (CS_BL)  at ($(cont.west)+(-0.30,\boxBottom)$);
    \coordinate (CS_TR)  at ($(plant.east)+(0.30,\boxTop)$);
    \coordinate (CS_Sep) at ($(CS_BL |- 0,\boxSep)$);

    \fill[gray200!40]
    (CS_BL) rectangle (CS_Sep -| CS_TR);

    \fill[gray200!80]
    (CS_Sep -| CS_BL) rectangle (CS_TR);

    \draw[control container]
    (CS_BL) rectangle (CS_TR);

    \draw[thick, draw=gray!80]
    (CS_Sep -| CS_BL) -- (CS_Sep -| CS_TR);

    \node[font=\normalsize]
    at ($(CS_BL)!0.5!(CS_Sep -| CS_TR)$)
    {Control System};

    \coordinate (Adv_BL)  at ($(attack.west)+(-0.30,\boxBottom)$);
    \coordinate (Adv_TR)  at ($(estimator.east)+(0.70,\boxTop)$);
    \coordinate (Adv_Sep) at ($(Adv_BL |- 0,\boxSep)$);

    \fill[my_rose_2!50, fill opacity=0.8]
    (Adv_BL) rectangle (Adv_Sep -| Adv_TR);

    \fill[my_rose_2!80, fill opacity=0.2]
    (Adv_Sep -| Adv_BL) rectangle (Adv_TR);

    \draw[adversary container]
    (Adv_BL) rectangle (Adv_TR);

    \draw[thick, draw=gray!80]
    (Adv_Sep -| Adv_BL) -- (Adv_Sep -| Adv_TR);

    \node[font=\normalsize]
    at ($(Adv_BL)!0.5!(Adv_Sep -| Adv_TR)$)
    {Adversary};

  \end{pgfonlayer}
\end{tikzpicture}
  }
  \caption{Problem setup. The block \(\Sigma_{\textrm{obs}}\) denotes the observer the adversary will design to achieve Objective 2.}
  \label{fig:problem_setup}
\end{figure}

In the following sections, we will present the main methodology to achieve Objectives \ref{obj:1}-\ref{obj:3}. The process of inducing local observability will be presented in Section \ref{sec:inducing_observability}, and the state estimation and stability analysis will be presented in Section \ref{sec:observer}.

\section{Inducing Local Observability}\label{sec:inducing_observability}
As explained in Section \ref{sec:problem_formulation_challenge}, the origin of systems with a quadratic output map is locally unobservable. To overcome this lack of local observability, the attack signal \(a\) will be designed using the knowledge of \(h(\cdot)\) and its gradient so that the linearization of the modified output map is nonzero at the origin. In the following analysis, the design procedure of the attack signal \(a\) is presented.

As discussed in Section~\ref{sec:problem_formulation}, the system in~\eqref{eq:sys_dynamics_clean} is unobservable because the linearization of its output map is zero at the origin. Although the analytical form of the output map's gradient is known, its true value cannot be calculated because it depends on the unavailable state \(z\). Therefore, the attacker must use an estimate, \(\hat{z}\), which will be provided by an estimator {designed} in Section \ref{sec:observer}. The attack signal \(a\) is designed by projecting the estimated gradient onto \(\mathbb{R}\) with the following projection vector \([\pi^\top, 0^\top]\), with $\pi \in \bbR^{n_p}$ such that
\begin{equation}
    \begin{aligned}
        a(\hat z) & := \begin{bmatrix} \pi^\top &0^\top\end{bmatrix}\nabla h(\hat z) = \bar H(\pi)\hat z,
    \end{aligned}
    \label{eq:attack_signal}
\end{equation}
where \(\bar H(\pi):=\begin{bmatrix}2\pi^\top Q_p & 0\end{bmatrix}\). The projection vector characerized by \(\pi\) is a design parameter that will be used to induce local observability.

The modified output defining \(\tilde y\) is therefore \(h(z) + \bar H(\pi) \hat z = h(z) + \bar H(\pi)(z+e)\), where \(e := \hat{z} - z\) is the state estimation error. The linearization of this new output map around the origin is now nonzero for a suitable choice of the projection vector \(\pi\).
Building on this, the linearization of the modified system in {\eqref{eq:system_dyn}} around the equilibrium, while treating the error \(e\) as input to the system, is characterized by the following respective system and output matrices:
\begin{align}
    \bar{F}(\pi) & := \frac{\partial \tilde f(z, a)}{\partial z}\bigg|_{z=0} =
    \underbrace{\begin{bmatrix}
                        A_p + 2 B_p D_c \pi^\top Q_p & B_p C_c \\
                        2 B_c \pi^\top Q_p           & A_c
                    \end{bmatrix}}_{= A + B \bar{H}(\pi)}, \label{eq:Fbar}                \\
    \bar{H}(\pi) & = \begin{bmatrix} 2\pi^\top Q_p & 0 \end{bmatrix}. \label{eq:Hbar}
\end{align}
The problem of inducing observability now reduces to designing a suitable {projection} vector $\pi$. The objective is therefore to establish a procedure for selecting the design vector $\pi \in \mathbb{R}^{n_p}$ to induce local observability.

This is achieved through a systematic approach: first, we derive algebraic conditions for the unobservable modes using the \gls{pbh} lemma, then prove that such modes cannot span the entire space $\bbR^{n_p}$, thereby ensuring the existence of a suitable {projection} vector. Next, we derive the unobservability conditions.

First, recall the \gls{pbh} test given in the following lemma.
\begin{lem}[\gls{pbh} test \cite{Hespanha2018}, Ch. 15]\label{lem:pbh}
    A linear time invariant pair $(F, H)$ is unobservable if and only if there exists a complex number $s \in \bbC$ and a non-zero vector ${w} \in \bbC^n$ such that
    \begin{align}
        H{w} = 0, \quad F{w} = s{w},
    \end{align}
    {which is equivalent to \(\textrm{rank}\left(\begin{bsmallmatrix}
            F - s I\\
            H
        \end{bsmallmatrix}\right) \neq n, \forall\,s \in \bbC.\)}
\end{lem}

\begin{cor}\label{cor:unobs_cond}
    The pair $(\bar{F}(\pi), \bar{H}(\pi))$, as given in \eqref{eq:Fbar} and \eqref{eq:Hbar}, is unobservable if and only if there exists a complex number $s \in \bbC$ and a non-zero vector ${w} = [{w}_p^\top, {w}_c^\top]^\top \in \bbC^{n_p+n_c}$ satisfying the following three conditions
    \begin{equation} \label{eq:pbh_combined}
        \begin{split}
            \pi^\top Q_p w_p & = 0,\\
            (sI - A_c) w_c &= 0,\\
            (sI - A_p) w_p &= B_p C_c w_c,
        \end{split}
    \end{equation}
    simultaneously.
\end{cor}
\begin{proof}
    From Lemma~\ref{lem:pbh}, unobservability requires $\bar F(\pi){w}=s{w}$ and $\bar H(\pi){w}=0$. The latter gives $\pi^\top Q_p{w}_p=0$. Substituting into the block form of $\bar F(\pi){w}=s{w}$ cancels the terms depending on $w_p$, leaving $(sI-A_p){w}_p=B_p C_c{w}_c$ and $(sI-A_c){w}_c=0$, which are the conditions in \eqref{eq:pbh_combined}.
\end{proof}
By this corollary, we reduce the question of unobservability to a set of algebraic conditions that subsequently provide a systematic procedure for inducing observability.
Next, we will show that choosing a vector $\pi$ to induce observability will be done by avoiding a particular finite set of hyperplanes in its design.
Additionally, we will show that such a projection vector \(\pi\) always exists. The details are given in the next lemma. It shows that an adversary can systematically choose a specific vector $\pi$ for this purpose. But first we state the following assumption.
\begin{assumption}[System regularity]\label{as:system_spectra}
    (i) The spectra of the plant and controller matrices are disjoint, i.e., $\sigma(A_p) \cap \sigma(A_c) = \varnothing$, and (ii) the term $B_p C_c \neq 0$.
\end{assumption}
Assumption \ref{as:system_spectra}-(i) is non-restrictive, requiring only that the plant \eqref{eq:plant} and the output-feedback dynamic controller \eqref{eq:controller} do not share the same poles. Assumption \ref{as:system_spectra}-(ii) is a fundamental requirement for output-feedback controllers as it gives the controller actuation authority. Assumption \ref{as:system_spectra} is crucial for the lemma below as it ensures that \((\lambda_{c,j}I-A_p)\) is invertible for every controller eigenvalue \(\lambda_{c,j}\) with \(j \in \{1, \dots, n_c\}\), so that the corresponding vector \(w_{p,j}=(\lambda_{c,j}I-A_p)^{-1}B_pC_cw_{c,j}\)
is uniquely defined. The associated constraint
\(\pi^\top Q_pw_{p,j}=0\) then defines one forbidden subspace in
\(\mathbb R^{n_p}\), thus ensuring that the set $\mathcal{H}_{C}$ in \eqref{eq:hyperplane_controller} below has exactly $n_c$ elements.

\begin{lem}
    \label{lem:unobservability}
    Under Assumption \ref{as:system_spectra}, the system defined by \eqref{eq:system_dyn} and \eqref{eq:attack_signal} is locally observable at the equilibrium ${z}={0}$ if and only if there exists \(\pi \in \bbR^{n_p}\) such that \(\pi \notin \Pi_{\textrm{unobs}}\) with
    \begin{equation} \label{eq:pi_unobs}
        \Pi_{\mathrm{unobs}} := \cH_P \cup \cH_C ,
    \end{equation}
    where
    \begingroup
    \small
    \begin{align}
        \cH_P & = \bigcup_{i=\{1,\dots,n_p\}} \left\{ \pi \in \bbR^{n_p} \mid \pi^\top Q_p w_{p,i} = 0 \right\}, \label{eq:hyperplane_plant} \\
        \cH_C & = \bigcup_{j=\{1,\dots,n_c\}} \left\{ \begin{aligned}
                                                           & \pi \in \bbR^{n_p} \mid                                               \\
                                                           & \pi^\top \big(Q_p(\lambda_{c,j} I - A_p)^{-1} B_pC_c w_{c,j}\big) = 0
                                                      \end{aligned} \right\}, \label{eq:hyperplane_controller}
    \end{align}
    \endgroup
    where $(\lambda_{p,i}, {w}_{p,i})$ and $(\lambda_{c,j}, {w}_{c,j})$ are the eigenpairs of the plant and controller matrices $A_p$, and $A_c$, respectively. Moreover, there always exists a vector ${\pi^\star} \in \bbR^{n_p} \setminus \Pi_{\textrm{unobs}}$ that renders the system \eqref{eq:system_dyn}, \eqref{eq:attack_signal} locally observable at the origin for $\pi = \pi^\star$.
\end{lem}

\begin{proof}
    See Appendix \ref{app:unobservability}.
\end{proof}

Hence, a suitable projection vector $\pi^\star$ always exists to meet Objective \ref{obj:1} for control system \eqref{eq:plant}, \eqref{eq:controller}, which satisfies Assumption \ref{as:system_spectra}. This, in turn, provides a foundation for the design of an observer that exploits this induced observability, which is developed in the next section.

\section{State Observer Design} \label{sec:observer}

Having established in Lemma~\ref{lem:unobservability} that the system {\eqref{eq:system_dyn} with \eqref{eq:attack_signal}} becomes locally observable for a proper choice of vector $\pi$, we now address the design of an observer that reconstructs the state $z$ from the specially designed (perturbed) output $\tilde{y}$. This is the step toward achieving the Objective~\ref{obj:2}, ensuring that the estimation error converges with a desired exponential rate.

\subsection{Observer structure}

We will synthesize a Luenberger-type observer. Let $\hat{z} := [\hat{x}_p^\top, \hat{x}_c^\top]^\top$ be the estimate of the state $z$. The observer is designed as follows:
\begin{equation} \label{eq:observer}
    \begin{split}
        \dot{\hat{z}}  = A \hat{z} & + B \big( \hat{z}^{\top} Q \hat{z} + 2\bar{H}(\pi)  \hat{z} \big) \\
        & + L \big( \hat{z}^{\top} Q \hat{z} + 2\bar{H}(\pi) \hat{z} - {\tilde y} \big)
    \end{split}
\end{equation}
where $L \in \bbR^{n \times 1}$ is the observer gain vector to be designed and recall that \(\tilde y\) is defined in \eqref{eqn:corrupted_signal}. The observer is designed to leverage the induced observability of the modified pair $(\bar{F}(\pi), \bar{H}(\pi))$. Thus, the additional factor $2$ multiplying $\bar{H}(\pi)\hat{z}$ in the nonlinear output model  $\hat{z}^{\top}Q\hat{z} + 2\bar{H}(\pi)\hat{z}$ in \eqref{eq:observer} ensures that the observability of $(\bar{F}(\pi), \bar{H}(\pi))$ {plays a role in the convergence of the estimation error system \(e\)}. More details are given in the following subsections.

Having designed the observer dynamics in~\eqref{eq:observer}, we next analyze the behavior of the estimation error. This analysis provides the foundation for establishing the convergence properties of the observer.

\subsection{Error dynamics}
Recall the definition of the estimation error $e = \hat z - z$. Subtracting the observer dynamics \eqref{eq:observer} from the system dynamics \eqref{eq:system_dyn} yields
\begin{equation}
    \begin{split} \label{eq:error_dynamics}
        \dot e = Ae &+ (B+L)(\hat{z}^\top Q \hat{z} - z^\top Q z) \\ &
        + (B+L)\bar{H}(\pi)\hat{z}.
    \end{split}
\end{equation}
Using the definition for \(\bar F(\pi)\) in \eqref{eq:Fbar} and by expanding the quadratic terms as
$\hat{z}^\top Q \hat{z} - z^\top Q z = (z + e)^\top Q (z + e) - z^\top Q z = 2z^\top Q e + e^\top Q e$, we arrive at the following compact form of the error dynamics:
\begin{align} \label{eq:error_dyn_final}
    \dot{e} = & \left(\bar{F}(\pi) + L\bar{H}(\pi)\right)e + (B+L)\bar{H}(\pi)z \nonumber \\
              & + (B+L)(2z^\top Q e + e^\top Q e) \,=:\, f_e(e, z).
\end{align}
The error dynamics in~\eqref{eq:error_dyn_final} describe the evolution of the estimation error $e$ which is coupled with the state \(z\). To analyze the local stability of the plant state and the observer convergence,
we now consider the combined dynamics of $z$ and $e$ as a single augmented system.
In this way, we can investigate the closed-loop stability and local observability properties under the designed attack signal.

\subsection{Achieving both local stability and local observability}
Consider the augmented system with state $\varphi = [z^\top, e^\top]^\top$ that has the dynamics,
\begin{equation}\label{eq:phi_full}
    \dot \varphi = \begin{pmatrix}
        Az + B\left(z^\top Q z + \bar H(\pi)(z+e)\right) \\
        f_e(e, z)
    \end{pmatrix} =: f_\varphi (\varphi),
\end{equation}
where recall that \(\bar H(\pi)\) is given in \eqref{eq:Hbar} and \(f_e(e, z)\) is given in \eqref{eq:error_dyn_final}.
The Jacobian of $f_\varphi(\varphi)$ at the origin $\varphi=0$ is
\begin{align} \label{eqn:aug_linearized_dyn}
    J_{\varphi} =
    \left[
        \begin{array}{c;{2pt/2pt}c}
            \bar{F}(\pi)      & B\bar{H}(\pi)                \\ \hdashline[2pt/2pt] \noalign{\vskip 2pt}
            (B+L)\bar{H}(\pi) & \bar{F}(\pi) + L\bar{H}(\pi)
        \end{array}
        \right].
\end{align}
To analyze the stability of the augmented system, we perform a similarity transformation on the Jacobian matrix $J_{\varphi}$ using the transformation matrix $T = \begin{bsmallmatrix} I & 0 \\ I & I \end{bsmallmatrix}$, which changes {its} coordinates from $[z^\top, e^\top]^\top$ to $[z^\top, \hat{z}^\top]^\top$. The resulting transformed matrix, $\tilde J_{\varphi} = T J_{\varphi} T^{-1}$, is block upper-triangular and is given below
\begin{equation} \label{eq:phi_Jacobian}
    \tilde{J}_{\varphi} =
    \left[
        \begin{array}{c;{2pt/2pt}c}
            A & B\bar{H}(\pi)                    \\ \hdashline[2pt/2pt] \noalign{\vskip 2pt}
            0 & \bar{F}(\pi) + (B+L)\bar{H}(\pi)
        \end{array}
        \right].
\end{equation}

Thus, the spectrum of the full augmented system's Jacobian, \(\sigma(J_{\varphi})\), is equivalent to the spectrum of its similarity transform, \(\tilde J_{\varphi}\), which is block-triangular. The eigenvalues of the Jacobian in \eqref{eq:phi_Jacobian} are the union of the eigenvalues of the diagonal blocks. The upper block is the matrix $A$, which is {Hurwitz by Assumption \ref{as:system_stability} (see Remark \ref{rem:system_stability})}. For the bottom-right block, since the pair $(\bar{F}(\pi), \bar{H}(\pi))$ has been made observable, a gain $\tilde{L} = (B+L)$ can be chosen (by proper design of $L$) to assign any desired stable spectrum for \(\bar F(\pi) + (B+L) \bar H(\pi)\).

Accordingly, both Objectives \ref{obj:2} and \ref{obj:3} are achieved. The obtained result guarantees maintaining the local exponential stability of the system \eqref{eq:system_dyn}-\eqref{eq:attack_signal} and local exponential convergence of the estimation error \eqref{eq:error_dyn_final}. However, in practical applications, local stability alone is insufficient and an estimate of the \gls{roa} is required \citep[pp. 312-314]{khalil_nonlinear_2002}. Consequently, we will next augment this result with an estimate of the \gls{roa}.

But first, we will present an intermediate analysis to show a condition under which \(\bar F(\pi)\) is Hurwitz, which is fundamental in finding an estimate of the \gls{roa}. Since \(\bar F(\pi) = A + B \bar H(\pi)\), it might seem that if the pair \((A, B)\) is controllable, then the matrix \(\bar F(\pi)\) can be made Hurwitz by the choice of \(\bar H(\pi)\). Recall from \eqref{eq:Hbar} that $\bar{H}(\pi) = [2\pi^\top Q_p\; 0]$, which has a structure that limits this proposition. We propose an alternative by scaling the choice of the projection vector $\pi$ by a scalar $\gamma$ such that the matrix $\bar{F}(\pi)$ remains Hurwitz. We show this in detail in the following lemma.

\begin{lem}\label{lem:bound}
    Consider the dynamics $\dot z = \bar F(\pi)z$ where we recall from \eqref{eq:Fbar} that  $\bar F(\pi) = A + B\bar H(\pi)$ with $\bar H(\pi)$ as defined in \eqref{eq:Hbar}. Let $\pi^\star \in \bbR^{n_p} \setminus \Pi_{\textrm{unobs}}$ be such that the pair $(\bar F (\pi^\star), \bar H (\pi^\star))$ is observable and define
    \begin{gather}
        \pi = \gamma \pi^\star \label{eq:pi_scaling}
    \end{gather}
    with the factor $\gamma$ satisfying
    \begin{gather}
        0 \;<\; \gamma \;<\; \frac{\lambda_{\min}(Y)}{4\,\|S B\|\,\|Q_p\pi^\star\|}, \label{eq:gamma_bound}
    \end{gather}
    where \(S\) solves \(A^\top S + S A \;=\; -Y\), with \(Y\succ0\). Then, \(\bar F(\pi)\) is Hurwitz. Moreover, the pair \(\bigl(\bar F(\pi),\bar H(\pi)\bigr) \) remains observable.
\end{lem}

\begin{proof}
    See Appendix \ref{app:bound}.
\end{proof}

Lemma \ref{lem:bound} allows for a scaling of the projection vector $\pi$ with $\gamma$ as in \eqref{eq:pi_scaling} while preserving the stability of the matrix $\bar{F}(\pi)$. This facilitates the estimation of the ROA of the augmented system \eqref{eq:phi_full} which is provided in the following lemma.

\begin{lem}[\gls{roa} estimate] \label{lem:roa}
    Consider the augmented system \eqref{eq:phi_full}. Let a projection vector $\pi = \gamma \pi^\star$ where $\pi^\star \in \bbR^{n_p} \setminus \Pi_{\textrm{unobs}}$, with $\Pi_{\textrm{unobs}}$ defined in \eqref{eq:pi_unobs}, and $\gamma$ chosen according to Lemma \ref{lem:bound}. Let an observer gain $L$ be chosen such that $\bar F(\pi)+L\bar H(\pi)$ is Hurwitz. Consider matrices $W_1, W_2 \in \mathbb{R}^{n \times n}$, and let $P_1, P_2 \succ 0$ satisfy the Lyapunov equations
    \begingroup
    \begin{equation}\label{eq:roa_lyap}
        \hspace{-0.5em}
        \begin{split}
            &\bar F(\pi)^\top P_1 + P_1 \bar F(\pi) = -W_1,  \\
            &\big(\bar F(\pi) + L\bar H(\pi)\big)^\top P_2 + P_2 \big(\bar F(\pi) + L\bar H(\pi)\big) = -W_2. \hspace{-1em}
        \end{split}
    \end{equation}
    \endgroup
    Define the constants $c_1, c_3(\pi), c_4 > 0$ as
    \begingroup
    \allowdisplaybreaks
    \begin{equation} \label{eq:roa_constants}
        \begin{split}
            c_1 &:= \frac{\min\{\sigma_{\min}(W_1), \sigma_{\min}(W_2)\}}{\max\{\sigma_{\max}(P_1), \sigma_{\max}(P_2)\}}, \\[\jot]
            c_3(\pi) &:= \frac{
                2\|P_1 B\bar H(\pi)\|
                + 2\|\bar H(\pi)^\top (B+L)^\top P_2\|
            }{
                \sqrt{\lambda_{\min}(P_1)\lambda_{\min}(P_2)}
            }, \\[\jot]
            c_4 &:= \frac{2\|P_1 B\|\|Q\|}{\lambda_{\min}(P_1)^{3/2}} + \frac{4\|P_2(B+L)\|\|Q\|}{\sqrt{\lambda_{\min}(P_1)}\lambda_{\min}(P_2)}  \\
            &\quad + \frac{2\|P_2(B+L)\|\|Q\|}{\lambda_{\min}(P_2)^{3/2}}.
        \end{split}
    \end{equation}
    \endgroup
    Suppose $c_2(\pi) := c_1 - c_3(\pi) > 0$, then for any $\delta \in (0, c_2(\pi))$, an estimate of the \gls{roa} is given by the sublevel set $\Omega_c$:
    \begin{equation} \label{eq:roa_set}
        \Omega_c := \left\{ \varphi \in \mathbb{R}^{2n} \,\middle|\,
        V(\varphi) \le
        \left(\frac{c_2(\pi) - \delta}{c_4}\right)^2
        \right\},
    \end{equation}
    of the Lyapunov function

    \begin{equation} \label{eq:roa_lap_fun}
        V(\varphi) = z^\top P_1 z + e^\top P_2 e.
    \end{equation}
    In the interior of $\Omega_c$, the trajectories of $\varphi$ satisfy $\dot{V}(\varphi) \le -\delta V(\varphi)$.
\end{lem}

\begin{proof}
    See Appendix \ref{app:roa}
\end{proof}
The derived sublevel set $\Omega_c$ is a conservative estimate of the true \gls{roa}. This conservatism arises from two primary sources: (i) the geometric limitation of approximating the true ROA with a quadratic Lyapunov function and (ii) the worst-case scenario when bounding the norms.

Despite this conservatism, the derived \gls{roa} can be tuned. The factors that serve as tuning knobs include the choice of the observer gain $L$,  and the choices of matrices \(W_1\) and \(W_2\).
\begin{remark}
    Lemma \ref{lem:roa} offers an initial estimate of the \gls{roa}. If the operational range, i.e., the set of expected initial conditions bounded by a polytope, of the system at hand is known, the approach in \cite{amato2007region} {can be adopted to obtain a less conservative estimate of the \gls{roa}}. This method addresses whether a predefined polytope \(\mathcal{P}\) of initial conditions containing the equilibrium point lies within the true \gls{roa} for quadratic systems of the form \eqref{eq:system_dyn}. By verifying the inclusion of $\mathcal{P}$ in the \gls{roa} we can construct a less conservative estimate of the \gls{roa}. This method applies directly to our system \eqref{eq:system_dyn}.
\end{remark}
Having established the individual components of each objective, we are now in a position to state the main result of this work. The preceding analysis has shown that the challenges of inducing observability and preserving stability can be systematically overcome. The following theorem unifies the preceding analysis.

\begin{thm} \label{thm:main}
    Consider the closed-loop system \eqref{eq:system_dyn}, with attack signal \eqref{eq:attack_signal} under Assumptions \ref{as:system_stability}, \ref{as:adversary_model}, and \ref{as:system_spectra}. Suppose \(\pi^\star \in \bbR^{n_p} \setminus \Pi_{\mathrm{unobs}}\), where \(\Pi_{\mathrm{unobs}}\) is defined in \eqref{eq:pi_unobs}. Let \(\pi = \gamma \pi^\star\) where \(\gamma > 0 \) satisfies \eqref{eq:gamma_bound}. Then for any \(\alpha>0\), there exists an observer gain \(L \in \bbR^{n\times 1}\) such that, for every \((z(0),\hat z(0))\in \Omega_c^z\),
    \[
        \|z(t)-\hat z(t)\|
        \le
        \kappa e^{-\alpha t}\|z(0)-\hat z(0)\|,
        \qquad t\ge 0,
    \]
    for some \(\kappa>0\), where
    \[
        \Omega_c^z :=
        \left\{
        (z,\hat z)\in\bbR^{2n}
        \ \middle|\
        V(z,\hat z-z)
        \le
        \left(\frac{c_2(\pi)-\delta}{c_4}\right)^2
        \right\},
    \]
    with \(\delta \in (0, c_2(\pi))\), \(c_2(\pi)\) and \(c_4\) as defined in \eqref{eq:roa_constants}, and $V(z,\hat{z}-z)= z^\top P_1z+(\hat{z}-z)^\top P_2(\hat{z}-z)$ defined in \eqref{eq:roa_lap_fun} with $P_1, P_2$ satisfying \eqref{eq:roa_lyap} for positive definite $W_1$ and $W_2$.
    Moreover, the closed-loop system \eqref{eq:system_dyn}, \eqref{eq:attack_signal} is LES.
\end{thm}

\begin{proof}
    From Lemma~\ref{lem:unobservability}, there always exists a vector $\pi^\star \in \mathbb{R}^{n_p}$ scaled by $\gamma$ following condition~\eqref{eq:gamma_bound} to give $\pi = \gamma \pi^\star$ that renders the pair $(\bar{F}(\pi), \bar{H}(\pi))$ observable. This allows the design of an observer as in~\eqref{eq:observer} with a gain $L$ selected to assign the eigenvalues of $\bar{F}(\pi) + (B+L)\bar{H}(\pi)$ arbitrarily such that the real parts of the eigenvalues are strictly negative, rendering the matrix Hurwitz. Consequently, according to Lemma \ref{lem:roa}, the estimation error system \eqref{eq:error_dyn_final} is LES with an \gls{roa} estimate $\Omega_c^z$. Hence, both diagonal blocks of $\tilde J_\varphi$ are Hurwitz ($A$ by Assumption~\ref{as:system_stability}, and $\bar F(\pi)+(B+L)\bar H(\pi)$ by choice of $L$). Thus, $\tilde J_\varphi$ is Hurwitz. Since the eigenvalues of $\tilde J_\varphi$ are the same as $J_\varphi$ (via a similarity transformation), $J_\varphi$ is also Hurwitz. Therefore, the system in \eqref{eq:phi_full} is LES.
\end{proof}

Thus, the existence of such a vector $\pi$ guarantees that the attack signal $a$ successfully induces local observability, satisfying Objective \ref{obj:1}. Furthermore, the established Hurwitz property of the Jacobian $\tilde J_\varphi$ ensures the local exponential stability of the augmented system \eqref{eq:phi_full}. This implies the local exponential convergence of the state estimation error system \eqref{eq:error_dyn_final}, thereby achieving both state estimation and preservation of closed-loop stability as required by Objectives \ref{obj:2} and \ref{obj:3}.

\section{Illustrative Example} \label{sec:example}
We now illustrate the design procedure on a fourth-order system, composed of a second-order plant ($n_p=2$) and a second-order controller ($n_c=2$). The source code used to generate these results can be found at {https://github.com/zmanaa/state-est-attack}.

The plant matrices \eqref{eq:plant} are given by
$
    A_p = \begin{bsmallmatrix} -6 & 2 \\ -5 & -1 \end{bsmallmatrix}
$, $B_p = \begin{bsmallmatrix} 1 \\ 1 \end{bsmallmatrix}$, and $Q_p = \frac{1}{2}\begin{bsmallmatrix} 1 & 0 \\ 0 & 1 \end{bsmallmatrix}$.
Moreover, the controller dynamics \eqref{eq:controller} is defined by the state-space model $(A_c, B_c, C_c, D_c)$ with
$
    A_c = \begin{bsmallmatrix} -7 & 4 \\ -8 & -7 \end{bsmallmatrix}$,
$B_c = \begin{bsmallmatrix} 1 \\ 1 \end{bsmallmatrix}$,
$C_c = \begin{bsmallmatrix} 1 & 0 \end{bsmallmatrix}$, and
$D_c = 1.0.$ {The given matrices \(A_p\) and \(A_c\) satisfy Assumption \ref{as:system_spectra} where \(\sigma(A_p) \cap \sigma(A_c) = \varnothing\) and \(B_p C_c \neq 0\).}
These are assembled into the $n=4$-dimensional augmented system \eqref{eq:sys_dynamics_clean} where the matrix $A$ is Hurwitz by construction, with $\sigma(A) = \{-3.5\pm1.94 j, -7\pm5.66 j\}$.
\begin{figure}[t]
    \centering
    \includegraphics[width=\linewidth]{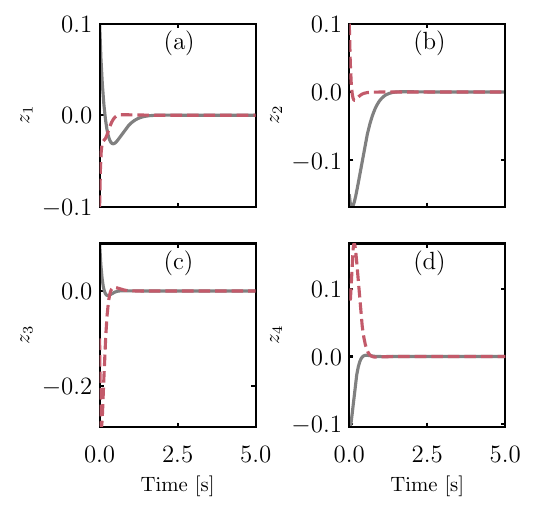}
    \caption{State estimation and the true system performance. True (\solidbox{my_gray}) and estimated (\dashboxx{my_rose}) trajectories for $z_1$ in (a), $z_2$ in (b), $z_3$ in (c), and $z_4$ in (d).}
    \label{fig:est}
\end{figure}

We choose $\pi=\gamma\pi^\star$ with $\pi^\star$ as defined in Lemma \ref{lem:unobservability} and
$\gamma$ as in \eqref{eq:gamma_bound}, with $Y = 0.2 I_n$. We find that any $\pi \neq 0$ renders the pair $(\bar F (\pi), \bar H(\pi))$ observable; hence, $\Pi_{\textrm{unobs}} = \{0\}$. We pick $\pi^\star = [1, -3]^\top \in \bbR^{n_p} \setminus \Pi_{\textrm{unobs}}$ and find $\gamma_{\max}=0.85$ according to Lemma \ref{lem:bound}. We conservatively choose $\gamma=0.9 \gamma_{\max} = 0.77$, yielding \(\pi = \gamma \pi^\star = [0.77,~-2.30]^\top\). This choice ensures that $\bar F(\pi)$ is Hurwitz and preserves the observability of $(\bar F(\pi),\bar H(\pi))$ according to Lemma \ref{lem:bound}, thereby satisfying Objective \ref{obj:1} (observability) and Objective \ref{obj:3} (stability). We then place the eigenvalues of $\bar F(\pi)+(B+L)\bar H(\pi)$ at $\{-9.5,\; -10.5,\; -11.5,\; -12.5\}$, which is the lower-right block in \eqref{eq:phi_Jacobian}, thereby satisfying Objective \ref{obj:2}.

We used the method of \cite{amato2007region} to verify that the box $\cP = [-0.5, 0.5]^8$ of initial conditions is contained in the \gls{roa}. Then we performed the simulation using the following initial conditions $z(0) = [0.1, -0.15, 0.1, -0.1]^\top$ and $\hat{z}(0) = [-0.1, 0.1, -0.1, 0.1]^\top$.
\begin{figure}
    \centering
    \includegraphics[width=\linewidth]{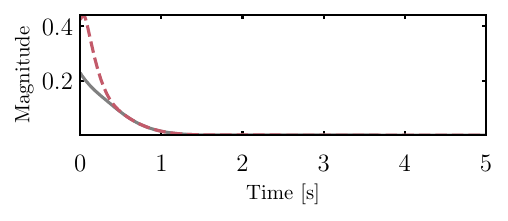}
    \caption{The norms of the state vector $\|z\|$ (\solidbox{my_gray}) and the estimation error $\|e\|$ (\dashboxx{my_rose}).}
    \label{fig:comp}
\end{figure}

The simulated performance  is shown in Figure \ref{fig:est} and Figure \ref{fig:comp}. The observed convergence of the state and the estimation error in Figure \ref{fig:comp} is consistent with the exponential upper bound established in Theorem \ref{thm:main}.

{With this example, we showed that Objectives \ref{obj:1}-\ref{obj:3} are achieved by choosing \(\pi\) according to Lemmas \ref{lem:unobservability} and \ref{lem:bound}. We also showed that the eigenvalues of the Jacobian of the observer system in the lower-right block in \eqref{eq:phi_Jacobian}, which is given by  \(\bar{F}(\pi) + (B+L)\bar{H}(\pi)\), can be chosen such that a desired local exponential decay rate \(\alpha\) is achieved.}
\section{Conclusion} \label{sec:conc}
We presented a systematic way to induce observability in closed-loop systems with linear dynamics and quadratic sensor outputs via a sensor attack. We showed that an appropriately designed attack signal can make an originally unobservable system locally observable while preserving closed-loop stability. A Luenberger-type observer was then developed to estimate both plant and controller states with guaranteed local exponential convergence and an estimate of the ROA. Future work will consider other controller structures and broader classes of nonlinear output maps.
\bibliography{ifacconf}

\appendix
\section{}
\subsection{Proof of Lemma \ref{lem:unobservability}}

\label{app:unobservability}
Starting with the \gls{pbh} test and considering the conditions derived in Corollary \ref{cor:unobs_cond}, the constraints on $\pi$ follow by considering separately the cases $w_c=0$ and $w_c\neq 0$.

\underline{Case 1: $w_c = 0$.}
Since $w \neq 0$, it follows that $w_p \neq 0$.
    {The} PBH conditions \eqref{eq:pbh_combined} reduce to  $(a) \;(sI - A_p)w_p = 0$ and $(b) \;\pi^\top Q_p w_p = 0$. Condition $(a)$ means that $s$ is an eigenvalue $\lambda_{p,i}$ of $A_p$ with eigenvector $w_{p,i}$. Condition $(b)$ then leads to the linear constraint
\[
    \pi^\top (Q_p w_{p,i}) = 0,
\]
which defines a hyperplane in $\bbR^{n_p}$ with normal vector $v_{p,i} := Q_p w_{p,i} \neq 0$.
Collecting these hyperplanes over all eigenpairs $(\lambda_{p,i},w_{p,i})$ of $A_p$ yields the set $\cH_P$ given in \eqref{eq:hyperplane_plant}.

\underline{Case 2: $w_c \neq 0$.}
From $(sI-A_c)w_c=0$, $s$ must be an eigenvalue $\lambda_{c,j}$ of $A_c$ with eigenvector $w_{c,j}$.
Substituting into the third equation in \eqref{eq:pbh_combined} gives the condition:
\begin{equation}
    (\lambda_{c,j} I - A_p)w_p = B_pC_c w_{c,j}. \label{eq:wp_sol}
\end{equation}
{Since \(\sigma(A_p)\) and \(\sigma(A_c)\) are disjoint and \(B_p C_c \neq 0\)} according to Assumption \ref{as:system_spectra}, the matrix $(\lambda_{c,j} I - A_p)$ is invertible,
so \eqref{eq:wp_sol} admits the unique nonzero solution
\[
    w_p = (\lambda_{c,j} I - A_p)^{-1} B_pC_c w_{c,j}.
\]
The remaining PBH condition $\pi^\top Q_p w_p = 0$ then defines the hyperplane
\[
    \left\{ \pi \in \bbR^{n_p} \;\middle|\;
    \pi^\top \big(Q_p(\lambda_{c,j} I - A_p)^{-1} B_pC_c w_{c,j}\big) = 0 \right\}.
\]
Taking the union over all eigenpairs $(\lambda_{c,j},w_{c,j})$ of $A_c$ yields the set $\cH_C$ given in \eqref{eq:hyperplane_controller}.

Now we prove that there always exists a \(\pi^\star \in \bbR^{n_p} \setminus \Pi_{\textrm{unobs}}\). First, note that the set of all projection vectors $\pi$ that do not induce observability, characterized by $\Pi_{\textrm{unobs}}$, is a finite union of hyperplanes in $\mathbb{R}^{n_p}$. Each hyperplane in \(\bbR^{n_p}\) is a proper subspace of dimension
\(n_p-1\). Since \(\bbR^{n_p}\) is a vector space over an infinite field, it cannot be
expressed as a finite union of proper subspaces \citep[pp.~196]{brzezinski2018galois}.

In turn, the finite union given in \eqref{eq:pi_unobs} cannot be the entire space $\bbR^{n_p}$. Therefore, the set of suitable vectors, $\Pi_{\text{obs}} = \bbR^{n_p} \setminus \Pi_{\text{unobs}}$, is always non-empty. Any vector ${\pi}^\star$ chosen from $\Pi_{\text{obs}}$ will render the system \eqref{eq:system_dyn}, \eqref{eq:attack_signal}, under Assumption \ref{as:system_spectra}, observable.

\subsection{Proof of Lemma \ref{lem:bound}} \label{app:bound}
Consider the candidate quadratic Lyapunov function \(U(z)=z^\top S z\),
where \(S\succ0\) is the solution to \({A^\top S + S A = -Y}\) with \(Y\succ0\),
which always has a solution since \(A\) is Hurwitz \cite[Thm. 4.6]{khalil_nonlinear_2002}.
The time derivative of \(U(z)\) along the solutions to ${\dot z = \bar F(\pi)z}$ is
\begin{align*}
    \dot U(z) & = z^\top\big[(A+B\bar H(\pi))^\top S + S(A+B\bar H(\pi))\big]z                \\
              & = z^\top(A^\top S + SA)z + z^\top(\bar H(\pi)^\top B^\top S + SB\bar H(\pi))z \\
              & = -z^\top Y z + z^\top(\bar H(\pi)^\top B^\top S + SB\bar H(\pi))z.
\end{align*}
Applying the Cauchy–Schwarz inequality and the sub-multiplicative property
of the induced 2-norm yields
\begin{align*}
    z^\top(\bar H(\pi)^\top B^\top S + SB\bar H(\pi))z
     & = 2z^\top (SB\bar H(\pi)) z            \\
     & \le 2\|SB\|\,\|\bar H(\pi)\|\,\|z\|^2.
\end{align*}
Substituting this bound back into the expression for \(\dot U(z)\) gives
\begin{align*}
    \dot U(z) & \le -z^\top Y z + 2\|S B\|\,\|\bar H(\pi)\|\,\|z\|^2                 \\
              & \le \big(-\lambda_{\min}(Y) + 2\|S B\|\,\|\bar H(\pi)\|\big)\|z\|^2.
\end{align*}
Therefore, \(\dot U(z)\) is negative definite and \(\bar F(\pi)\) is Hurwitz whenever \(\pi\) is chosen such that
\begin{equation}
    \label{eq:key_bound_S}
    2\|S B\|\,\|\bar H(\pi)\| \;<\; \lambda_{\min}(Y).
\end{equation}
More explicitly, since
\(\|\bar H(\pi)\|=\|[\,2\pi^\top Q_p\;0\,]\|=2\|Q_p\pi\|\), and \(\pi = \gamma \pi^\star\) according to \eqref{eq:pi_scaling},
\eqref{eq:key_bound_S} becomes
\begin{equation}
    2\|S B\|\cdot 2\gamma\|Q_p\pi^\star\| < \lambda_{\min}(Y),
\end{equation}
which is satisfied by the choice of \(\gamma\) in~\eqref{eq:gamma_bound}.

Since \(\gamma>0\) only rescales \(\pi^\star\), the vector
\(\pi=\gamma\pi^\star\) remains outside \(\Pi_{\mathrm{unobs}}\).
Therefore, the pair \((\bar F(\pi),\bar H(\pi))\) remains observable.

\subsection{Proof of Lemma \ref{lem:roa}} \label{app:roa}
We first show that there exist positive definite matrices \(P_1, P_2 \succ 0\) that satisfy \eqref{eq:roa_lyap}
for some \(W_1, W_2 \succ 0\). According to \cite[Thm. 4.6]{khalil_nonlinear_2002}, this holds when \(\bar F(\pi)\) and
\(\bar{F}(\pi)+L\bar{H}(\pi)\) are Hurwitz, which are guaranteed by Lemma \ref{lem:bound}. Then, the time derivative of \eqref{eq:roa_lap_fun} along the trajectories of \eqref{eq:phi_full} is
\begin{equation} \label{eq:lyap_deriv_correct}
    \dot{V}(\varphi) = -z^\top W_1 z - e^\top W_2 e + \Phi_{\textrm{cross}}(\varphi) + \Phi_{\textrm{honl}}(\varphi),
\end{equation}
where $\Phi_{\textrm{cross}}(\varphi)$ and $\Phi_{\textrm{honl}}(\varphi)$ are given by
\begin{align*}
    \Phi_{\textrm{cross}}(\varphi) = & \ 2z^\top P_1 B\bar{H}(\pi)e + 2z^\top \bar{H}(\pi)^\top(B+L)^\top P_2 e, \\
    \Phi_{\textrm{honl}}(\varphi) =  & \ 2z^\top P_1 B (z^\top Q z) + 2e^\top P_2(B+L)(2z^\top Q e)              \\
                                     & + 2e^\top P_2(B+L)(e^\top Q e).
\end{align*}
To bound the first two terms in~\eqref{eq:lyap_deriv_correct}, note that
\( z^\top W_1 z \ge \sigma_{\min}(W_1)\|z\|^2,\) and \( e^\top W_2 e \ge \sigma_{\min}(W_2)\|e\|^2.\)
Let $w_{\min} := \min\{\sigma_{\min}(W_1),\,\sigma_{\min}(W_2)\}$, which yields
\begin{equation}
    z^\top W_1 z + e^\top W_2 e
    \ge w_{\min}\big(\|z\|^2+\|e\|^2\big).
    \label{eq:Wmin_bound}
\end{equation}
Similarly, using the upper bound on the Lyapunov function in \eqref{eq:roa_lap_fun}, we have
\(
V(\varphi)
\le p_{\max}\big(\|z\|^2+\|e\|^2\big),\) with \(p_{\max} := \max\{\sigma_{\max}(P_1),\,\sigma_{\max}(P_2)\}.\)
Combining~\eqref{eq:Wmin_bound} with this upper bound gives
\begin{equation}
    z^\top W_1 z + e^\top W_2 e
    \ge \frac{w_{\min}}{p_{\max}}\,V(\varphi).
    \label{eq:CL_intermediate}
\end{equation}
Hence, the quadratic term in~\eqref{eq:lyap_deriv_correct} satisfies
\begin{equation}
    -\,z^\top W_1 z - e^\top W_2 e
    \le -\,c_1 V(\varphi),
    \label{eq:CL}
\end{equation} where \(c_1 := \frac{w_{\min}}{p_{\max}}\) as given in \eqref{eq:roa_constants}.

Next, we bound the remaining terms in \eqref{eq:lyap_deriv_correct}.
First, $|\Phi_{\mathrm{cross}}(\varphi)| \le c_{\mathrm{cross}}\|z\|\|e\|$,
where $c_{\mathrm{cross}} := 2\|P_1 B\bar{H}(\pi)\|$
\penalty0
$+\,2\|\bar{H}(\pi)^\top(B+L)^\top P_2\|$.
Using the relations
$\|z\| \le \sqrt{V(\varphi)/\sigma_{\min}(P_1)}$
and
$\|e\| \le \sqrt{V(\varphi)/\sigma_{\min}(P_2)}$,
we obtain
$|\Phi_{\mathrm{cross}}(\varphi)| \le c_3(\pi)V(\varphi)$,
where $c_3(\pi)$ is defined in \eqref{eq:roa_constants}.

Next, to bound the higher-order terms, we first apply norm inequalities to $\Phi_{\textrm{honl}}(\varphi)$ and obtain
\begin{align*}
    |\Phi_{\textrm{honl}}(\varphi)| \le & \ 2\|P_1 B\|\|Q\|\|z\|^3         \\
                                        & + 4\|P_2(B+L)\|\|Q\|\|z\|\|e\|^2 \\
                                        & + 2\|P_2(B+L)\|\|Q\|\|e\|^3.
\end{align*}
Using the relations $\|z\| \le \sqrt{V(\varphi)/\lambda_{\min}(P_1)}$ and $\|e\| \le \sqrt{V(\varphi)/\lambda_{\min}(P_2)}$, we substitute these into the inequality to factor out $V(\varphi)^{3/2}$
\begin{align*}
    |\Phi_{\textrm{honl}}(\varphi)| \le & \ \Bigg[ \frac{2\|P_1 B\|\|Q\|}{\lambda_{\min}(P_1)^{3/2}} + \frac{4\|P_2(B+L)\|\|Q\|}{\sqrt{\lambda_{\min}(P_1)}\lambda_{\min}(P_2)} \\
                                        & \quad + \frac{2\|P_2(B+L)\|\|Q\|}{\lambda_{\min}(P_2)^{3/2}} \Bigg] V(\varphi)^{3/2}.
\end{align*}
This simplifies to $|\Phi_{\textrm{honl}}(\varphi)| \le c_4 V(\varphi)^{3/2}$, where $c_4$ is given in \eqref{eq:roa_constants}.
Combining these bounds, we obtain
\begin{align}\label{eq:v_dot_bound}
    \dot{V}(\varphi) & \le -c_1 V(\varphi) + |\Phi_{\textrm{cross}}(\varphi)| + |\Phi_{\textrm{honl}}(\varphi)| \nonumber \\
                     & \le -c_2(\pi) V(\varphi) + c_4 V(\varphi)^{3/2}.
\end{align}

Since the trajectory starts within \(\Omega_c\), it follows that \(V(\varphi) \leq \left(\frac{c_2(\pi) - \delta}{c_4}\right)^2,\)
which in turn implies \(c_4 \sqrt{V(\varphi)} \leq c_2(\pi) - \delta.\) Substituting this into \eqref{eq:v_dot_bound} yields \(\dot V(\varphi) \leq \left(-c_2(\pi) + c_4 \sqrt{V(\varphi)} \right) V(\varphi) \leq -\delta V(\varphi)\), completing the proof.

\end{document}